\documentclass{amsart}

\usepackage{amssymb,amsmath}
\usepackage{color}
\usepackage{enumerate}
\usepackage[hdivide={3cm,*,3cm}]{geometry}

\title{On the family of measurable sets having the upper positive density}
\author{Jacek Hejduk \& Renata Wiertelak \& W{\l}adys{\l}aw Wilczy\'{n}ski}

\address{Faculty of Mathematics and Computer Science,
      University of {\L}\'{o}d\'{z},
      Banacha 22, PL-90-238 {\L}\'{o}d\'{z}, Poland}
\email{jacek.hejduk@wmii.uni.lodz.pl\\
renata.wiertelak@wmii.uni.lodz.pl\\
wladyslaw.wilczynski@wmii.uni.lodz.pl}

\date{}

\newtheorem{theorem}{Theorem}

\newtheorem{property}{Property}
\newtheorem{lemma}{Lemma}
\newtheorem{corollary}{Corollary}

\newtheorem{remark}{Remark}

\begin{document}
\maketitle

Declarations of interest: none

This research did not receive any specific grant from funding agencies in the public, commercial, or not-for-profit sectors.

\begin{abstract}
The essence of the density topology lies in the family of Lebesgue measurable sets where each point of a set is a density point of that set. The motivation of this work is to investigate the family of measurable sets for which, at every point within a set belonging to this family, the upper density of that set is positive. We obtain a strong generalized topology, and its essential properties are demonstrated in comparison with those of the classical density topology.
\end{abstract}

\section{Introduction}

The density topology is a significant concept in real analysis. Its origins lie in the work of Haupt and Pauc (1952, see \cite{hw-hp}) and it has been extensively studied by Goffman and Waterman (1961, see \cite{hw-gw}), Goffman, Neugebauer, and Nishiura (1961, see \cite{hw-gnn}), and Tall (1976, see \cite{hw-t}). This paper explores the properties of the family of Lebesgue measurable sets where, at every point within a given set, the upper density is positive. We will demonstrate that this family forms a generalized topology.

A generalized topology on a nonempty set \(X\) is defined as a family \(\gamma\) of subsets of \(X\) such that \(\emptyset \in \gamma\) and \(\bigcup_{t \in T} G_t \in \gamma\) whenever \(\{G_t : t \in T\} \subset \gamma\). The pair \((X, \gamma)\) is referred to as a generalized topological space. If \(X \in \gamma\), then \((X, \gamma)\) is called a strong generalized topological space.

The notion of a generalized topology was introduced independently by several mathematicians, with E.H. Moore presenting this concept in 1910 in \cite{hw-m}. \'{A}. Cs\'{a}sz\'{a}r also revisited the concept of generalized topology in \cite{hw-c} while investigating generalized open sets in topological spaces.

This paper explores the properties of open sets, closed sets, nowhere dense sets, and first category sets within the defined generalized topology. Several results are analogous to those of the classical density topology, but differences arise concerning approximative continuity. Notably, continuity with respect to the generalized density topology does not imply the Darboux property. The regularity and normality of this generalized topology are established, while the complete regularity remains an open question.

Throughout this paper, we will use the following standard notation: \(\mathbb{R}\) denotes the set of real numbers, \(\mathbb{N}\) the set of natural numbers, \(\mathcal{L}\) the family of Lebesgue measurable subsets of \(\mathbb{R}\), and \(\mathbb{L}\) the family of Lebesgue null sets. The Lebesgue measure is denoted by \(\lambda\), and the outer Lebesgue measure by \(\lambda^*\).

A set \(B \subset \mathbb{R}\) is defined as a measurable hull of a set \(A \subset \mathbb{R}\) if \(B \in \mathcal{L}\), \(A \subset B\), and for every \(C \in \mathcal{L}\) such that \(C \subset B \setminus A\), we have \(C \in \mathbb{L}\).

A set \(B \subset \mathbb{R}\) is defined as a measurable kernel of a set \(A \subset \mathbb{R}\) if \(B \in \mathcal{L}\), \(B \subset A\), and for every \(C \in \mathcal{L}\) such that \(C \subset A \setminus B\), we have \(C \in \mathbb{L}\).

\section{Properties of the operator \(\Phi^+\)}

For \(A \in \mathcal{L}\) and \(x \in \mathbb{R}\), we define the upper density \(\bar{D}(A, x)\) and the density \(D(A, x)\) of \(A\) at \(x\) as:
\[
\bar{D}(A,x)=\limsup_{h\to 0^+} \frac{\lambda \left( A\cap [x-h,x +h]\right) }{2h}
\]
\[
D(A,x)=\lim_{h\to 0^+} \frac{\lambda \left( A\cap [x-h,x +h]\right) }{2h}.
\]
Observe that \(\bar{D}(A \cup B, x) \leq \bar{D}(A, x) + \bar{D}(B, x)\) for all \(A, B \in \mathcal{L}\).
Now, let us define the operators:
\[
\Phi^+ (A)=\{ x\in \mathbb{R} : \bar{D}(A,x)> 0 \} 
\] 
\[
\Phi_d (A)=\{ x\in \mathbb{R} : D(A,x) =1 \} .
\] 
It is well-known that the family 
\(
\mathcal{T}_{d} =\left\{ A\in \mathcal{L}: A\subset \Phi_d(A)\right\}
\)
forms a topology on \(\mathbb{R}\), called the density topology (see \cite{hw-w}).

Firstly, we examine the properties of the operator \(\Phi^+\).

\begin{property}\label{prop}
\begin{enumerate}[1.]
 \item $\Phi^+ (\emptyset) = \emptyset,\quad \Phi^+ (\mathbb{R}) = \mathbb{R}$;
 \item $\forall _{A \in \mathcal{L}} \forall _{B \in \mathcal{L}} \quad A\triangle B \in \mathbb{L} \Rightarrow \Phi^+ (A) = \Phi^+(B)$;
 \item $\forall _{A \in \mathcal{L}} \forall _{B \in \mathcal{L}} \quad A\subset B \Rightarrow \Phi^+ (A) \subset \Phi^+ (B)$;
 \item $\forall _{A \in \mathcal{L}} \forall _{B \in \mathcal{L}}\quad \Phi^+ (A \cap B) \subset \Phi^+(A) \cap \Phi^+(B)$;
 \item $\forall _{A \in  \mathcal{L}} \quad \Phi^+ (A)\triangle A \in \mathbb{L}$;
 \item $\forall _{A \in  \mathcal{L}} \quad \Phi^+ (A)\triangle \Phi _{d}(A) \in \mathbb{L}$;
 \item $\forall _{A \in \mathcal{L}} \forall _{B \in \mathcal{L}} \forall _{x \in  \mathbb{R}} \quad  x\in \Phi _{d}(A)  \Rightarrow \bar{D}(A\cap B,x)=\bar{D}(B,x)$;
 \item $\forall _{A \in \mathcal{L}} \forall _{B \in \mathcal{L}}\quad  \Phi _{d}(A) \cap \Phi^+(B) \subset \Phi^+ (A \cap B)$.
\end{enumerate}
\end{property}

\begin{proof}
Conditions 1-4 follow directly from the definition of the operator \(\Phi^+\).

 To prove 5, let \(A\in \mathcal{L}\). Then
\[
A\setminus \Phi^+(A)\subset A\cap \Phi _{d}(\mathbb{R} \setminus A).
\]
By the Lebesgue Density Theorem, we have that 
\((\mathbb{R} \setminus A)\triangle \Phi _{d}(\mathbb{R} \setminus A)\in \mathbb{L}\).
Hence \(A\cap \Phi _{d}(\mathbb{R} \setminus A) \in \mathbb{L}\), and consequently \(A\setminus \Phi^+(A)\in \mathbb{L}\).
Moreover,
\[
 \Phi^+(A) \setminus A=\Phi^+(A) \cap (\mathbb{R} \setminus A)\subset (\mathbb{R} \setminus \Phi _{d}(\mathbb{R}\setminus A)) \cap (\mathbb{R} \setminus A) \in \mathbb{L} .
\]
This implies that \(\Phi^+ (A)\triangle A \in \mathbb{L} \).

For 6 we have
\[
\Phi^+ (A)\triangle \Phi _{d}(A)=(A \triangle \Phi^+ (A)) \triangle( A\triangle \Phi _{d}(A)) 
\] 
for every \(A\in \mathcal{L}\). By 5 and the Lebesgue Density Theorem we obtain that 
\(\Phi^+ (A)\triangle \Phi _{d}(A) \in \mathbb{L}\).

Now we prove 7. Let \(A\), \(B\in \mathcal{L}\) and \(x\in \Phi _{d} (A)\). Then
\[
\bar{D}(B,x) =\bar{D}((A\cap B)\cup (B\setminus A),x)\leq \bar{D}(A\cap B,x) + \bar{D}(B\setminus A,x)=\bar{D}(A\cap B,x) 
\]
Clearly, \(\bar{D}(A\cap B,x) \leq \bar{D}(B,x)\).
Therefore \(\bar{D}(A\cap B,x)=\bar{D}(B,x)\).

For 8, let \(A\), \(B\in \mathcal{L}\). If \(x\in \Phi _{d} (A) \cap \Phi^+(B)\), then by 7 we conclude that \(\bar{D}(A\cap B,x)>0\).
Hence \(x\in \Phi^+(A\cap B)\). This completes the proof.
\end{proof}

\begin{remark}\label{topol}
For \(A = (0, 1)\) and \(B = [0, 1]\), we have that \(\Phi _{d}(A) \cap \Phi^+(B) = (0, 1)\), whereas \(\Phi^+ (A \cap B) = [0, 1]\).
Therefore the reverse inclusion in property 8 does not hold.
\end{remark}

\begin{theorem}
For any set \(A \in \mathcal{L}\), \(\Phi^+(A) \in G_{\sigma \delta}\).
\end{theorem}

\begin{proof}
Observe that:
\[
\Phi^+ (A)=\bigcup_{m\in \mathbb{N}}\left\{ x\in \mathbb{R} : \forall_{k \in  \mathbb{N}} \; \exists_{n \in  \mathbb{N} , n>k} \; \frac{\lambda \left( A\cap [x-\frac{1}{n},x +\frac{1}{n}]\right) }{\frac{2}{n}} >\frac{1}{m} \right\} .
\] 
Indeed, if 
\[
\exists_{m\in \mathbb{N}} \; \forall_{k \in  \mathbb{N}} \; \exists_{n \in  \mathbb{N} , n>k} \; \frac{\lambda \left( A\cap [x-\frac{1}{n},x +\frac{1}{n}]\right) }{\frac{2}{n}} >\frac{1}{m}  
\]
then \(\limsup_{h\to 0^+} \frac{\lambda \left( A\cap [x-h,x +h]\right) }{2h} \geq \frac{1}{m}\). Therefore \(x\in \Phi^+ (A)\).

Conversely, assume that 
\[
\forall_{m\in \mathbb{N}} \; \exists_{k \in  \mathbb{N}} \; \forall_{n \in  \mathbb{N} , n>k} \; \frac{\lambda \left( A\cap [x-\frac{1}{n},x +\frac{1}{n}]\right) }{\frac{2}{n}} \leq \frac{1}{m}
\]
Hence \(\lim\limits_{n\to \infty} \frac{\lambda \left( A\cap [x-\frac{1}{n},x +\frac{1}{n}]\right) }{\frac{2}{n}}=0\). This means that \(x\in \Phi _{d}(\mathbb{R} \setminus A)\).
Thus \(\limsup_{h\to 0^+} \frac{\lambda \left( A\cap [x-h,x +h]\right) }{2h}=0\). Therefore \(x\notin \Phi^+ (A)\).

Summarizing the equivalences, we have
\[
\Phi^+ (A)=\bigcup_{m\in \mathbb{N}} \; \bigcap_{k \in  \mathbb{N}} \; \bigcup_{n \in  \mathbb{N} , n>k} \; \left\{ x\in \mathbb{R} : \frac{\lambda \left( A\cap [x-\frac{1}{n},x +\frac{1}{n}]\right) }{\frac{2}{n}} >\frac{1}{m} \right\} .
\] 
Since the set \( \left\{ x\in \mathbb{R} : \frac{\lambda \left( A\cap [x-\frac{1}{n},x +\frac{1}{n}]\right) }{\frac{2}{n}} >\frac{1}{m} \right\}\) is open in the natural topology, we conclude that \(\Phi^+ (A) \in G_{\sigma \delta}\).
\end{proof}

\section{Properties of the family \(\mathcal{T}^+\)}
Let us define the family \(\mathcal{T}^+ =\left\{ A\in \mathcal{L}: A\subset \Phi^+ (A)\right\}\).
Let $A=\bigcup_{n=1}^{\infty} (a_n, b_n) \subset (0,1)$ be right-sided interval set such that $0<\bar{D}(A,0)<1$.
\begin{theorem}\label{minus}
The family \(\mathcal{T}^+\) is a strong generalized topology such that \(\mathcal{T}_{d}\subset \mathcal{T}^+ \subset \mathcal{L}\).
\end{theorem}

\begin{proof}
By Property \ref{prop} we have \(\emptyset, \mathbb{R}\in \mathcal{T}^+\).
Let \(\{ A_t\}_{t\in T}\subset \mathcal{T}^+\). 
Since the pair \((\mathcal{S}_{\mu} , \mathcal{I}_{\mu})\) has the hull property, there exists a measurable kernel \(B\) of the set \(\bigcup_{t\in T}A_t\).
Then \((B\cap A_t)\triangle A_t \in \mathbb{L}\) for every \(t\in T\), and 
\[
B\subset \bigcup_{t\in T}A_t \subset \bigcup_{t\in T}\Phi^+(A_t)=\bigcup_{t\in T}\Phi^+(B\cap A_t) \subset \Phi^+(B) .
\]
By condition 5 of Property \ref{prop}, \(\Phi^+ (B)\setminus B \in \mathbb{L}\). Hence \(\bigcup_{t\in T }A_t\in \mathcal{L}\). 
By the monotonicity of \(\Phi^+\), we obtain \(\bigcup_{t\in T}A_t\subset \bigcup_{t\in T}\Phi^+ (A_t) \subset \Phi^+(\bigcup_{t\in T}A_t)\).
Clearly \(\mathcal{T}_{d}\subset \mathcal{T}^+\).
\end{proof}

\begin{remark}
The family \(\mathcal{T}^+\) is not a topology because it is not closed under finite intersections. 
For example, \((0,1] \in \mathcal{T}^+\) and \([1,2] \in \mathcal{T}^+\), but their intersection 
\[
(0,1]\cap [1,2]=\{ 1\} \notin \mathcal{T}^+.
\]
\end{remark}

\begin{remark}\label{rem3}
By Property  \ref{prop}.8, if \(A\in \mathcal{T}^+\) and \(B\in \mathcal{T}_{d}\), then \(A\cap B\in \mathcal{T}^+\).
\end{remark}

\begin{property}\label{dwa}
For every \(A\in \mathcal{L}\), the following conditions hold:
\begin{enumerate}[1.]
 \item \( \Phi^+ ( \Phi^+ (A))= \Phi^+ (A)\);
 \item \(\Phi^+ (A \cap \Phi^+ (A))= \Phi^+ (A)\);
 \item \((A \cap \Phi^+ (A)) \in \mathbb{L} \iff A\in \mathbb{L}\).
\end{enumerate}
\end{property}

\begin{proof}
Conditions 1 and 2 follow from conditions 2 and 5 of Property \ref{prop}.
It remains to prove 3.
Since \((A \cap \Phi^+ (A)) \triangle A \in \mathbb{L}\), we have \(A\in \mathbb{L}\).
The inverse implication is obvious.
\end{proof}

In light of this last result, we obtain the following property.

\begin{property}\label{topol}
For every \(A\in \mathcal{L}\), the sets \(\Phi^+ (A)\) and \(A \cap \Phi^+ (A)\) are \(\mathcal{T}^+\)-open.
\end{property}

\begin{theorem}\label{wnet}
For every $ A\subset \mathbb{R}$
$$
int_{\mathcal{T}^+} (A)=A\cap \Phi^+(B),
$$
where $B$ is measurable kernel of $A$.
\end{theorem}

\begin{proof}
Let \(A \subset \mathbb{R}\) and \(B\) be a measurable kernel of \(A\). We will show that \(\operatorname{int}_{\mathcal{T}^+} (A)\subset A\cap \Phi^+(B)\).
Fix \(x \in \operatorname{int}_{\mathcal{T}^+} (A)\). Then there exists \(V\in \mathcal{T}^+\) such that \(x\in V\subset A\). Hence \(x\in \Phi^+ (V)\).
By the definition of a measurable kernel, \(V\setminus B\in \mathbb{L}\). Consequently,
\[
\Phi^+ (V)= \Phi^+ ((V\cap B)\cup (V\setminus B))= \Phi^+ (V\cap B) \subset \Phi^+ (B).
\]
This implies that \(x\in A\cap \Phi^+(B)\).

Now assume that \(x\in A\cap \Phi^+(B)\). Since \(B\cup \{ x\} \subset A\) and \(\Phi^+ (B\cup \{ x\}) =\Phi^+ (B)\),
it follows that \(x\in \Phi^+ (B\cup \{ x\})\).
Let \(V= (B\cup \{ x\})\cap \Phi^+ (B\cup \{ x\})\). Then \(V\in \mathcal{T}^+\) and \(x\in V\subset A\).
Therefore \(x \in \operatorname{int}_{\mathcal{T}^+} (A)\).
\end{proof}

\begin{corollary}
Let \(A\in \mathcal{L}\). Then
\[
\operatorname{int}_{\mathcal{T}^+} (A)=\emptyset \iff A\in \mathbb{L} .
\]
\end{corollary}

\begin{proof} 
It suffices to demonstrate that \(\operatorname{int}_{\mathcal{T}^+} (A)=\emptyset\) implies \(A\in \mathbb{L}\).
By Theorem \ref{wnet} we have \(\operatorname{int}_{\mathcal{T}^+} (A)=A\cap \Phi^+(A)\).
By Property \ref{dwa} we conclude that \(A\in \mathbb{L}\).
\end{proof}

The above property does not hold for an arbitrary set \(A\subset \mathbb{R}\).
If \(A\) is a Bernstein set (see \cite{hw-o}), then \(\operatorname{int}_{\mathcal{T}^+} (A)=\emptyset\), but \(A\notin \mathbb{L}\).

The next theorem provides a characterization of the regular open sets in the space \(\langle \mathbb{R}, \mathcal{T}^+ \rangle\). A set \(A\in \mathcal{T}^+\) is \(\mathcal{T}^+\)-regular open if \(A=\operatorname{int}_{ \mathcal{T}^+} ( \operatorname{cl}_{ \mathcal{T}^+} (A))\).

\begin{theorem}
A set \(A\in \mathcal{T}^+\) is \(\mathcal{T}^+\)-regular open if and only if \(A= \operatorname{cl}_{ \mathcal{T}^+} (A) \cap \Phi^+ (A)\).
\end{theorem}

\begin{proof} 
The set \(A\in \mathcal{T}^+\) is \(\mathcal{T}^+\)-regular open if and only if
\[
A=\operatorname{int}_{ \mathcal{T}^+} ( \operatorname{cl}_{ \mathcal{T}^+} (A)) =\operatorname{cl}_{ \mathcal{T}^+} (A) \cap \Phi^+ (\operatorname{cl}_{ \mathcal{T}^+} (A) )
\]
Since \((\mathbb{R} \setminus A)\triangle \Phi^+(\mathbb{R} \setminus A) \in \mathbb{L}\), we have \(A\triangle (\mathbb{R} \setminus \Phi^+(\mathbb{R} \setminus A)) \in \mathbb{L}\).
Moreover, 
\[
\operatorname{cl}_{ \mathcal{T}^+} (A) =\mathbb{R} \setminus (\operatorname{int}_{ \mathcal{T}^+} (\mathbb{R} \setminus A))=\mathbb{R} \setminus ( (\mathbb{R} \setminus A) \cap \Phi^+(\mathbb{R} \setminus A))=A \cup (\mathbb{R} \setminus \Phi^+(\mathbb{R} \setminus A) ) .
\]
Thus \(A\triangle \operatorname{cl}_{ \mathcal{T}^+} (A)  \in \mathbb{L}\), and consequently, \(\Phi^+ (\operatorname{cl}_{ \mathcal{T}^+} (A) )=\Phi^+ (A)\).
\end{proof}

It is worth noting that in the space \(\langle \mathbb{R}, \mathcal{T}_{d} \rangle\), regular open sets are characterized by the equality \(A=\Phi _{d}(A)\) (see \cite{hw-w}).

\begin{corollary}
If \(A\in \mathcal{L}\) and \(A=\Phi^+ (A)\), then \(A\) is \(\mathcal{T}^+\)-regular open.
\end{corollary}

\begin{proof}
Indeed, clearly \(A \in \mathcal{T}^+\) and
\[
\operatorname{cl}_{ \mathcal{T}^+} (A) \cap \Phi^+ (A)= \operatorname{cl}_{ \mathcal{T}^+} (A) \cap A=A ,
\]
thus \(A\) is regular open.
\end{proof}

Note that the converse of the above corollary does not hold.
For \(A=[0,1)\), we have \(A=\operatorname{int}_{ \mathcal{T}^+} ( \operatorname{cl}_{ \mathcal{T}^+} (A))\), but \(A\neq \Phi^+ (A)=[0,1]\).

Now we will characterize the closure in the space $\langle \mathbb{R}, \mathcal{T}^+ \rangle$.

\begin{theorem}
For every $ A\subset \mathbb{R}$
$$
cl_{\mathcal{T}^+} (A)=A\cup \left\{ x\in \mathbb{R} :\lim_{h\to 0^+} \frac{\lambda^* \left( A\cap [x-h,x +h]\right) }{2h}=1\right\} .
$$
\end{theorem}

\begin{proof} 
$$
cl_{\mathcal{T}^+} (A)=\mathbb{R} \setminus int_{\mathcal{T}^+} (\mathbb{R} \setminus A) =A\cup (\mathbb{R}\setminus \Phi^+ (B)),
$$
where $B$ is measurable kernel of $\mathbb{R}\setminus A$.
Since $\mathbb{R}\setminus \Phi^+ (B)=\Phi _{d}(\mathbb{R}\setminus B)$, we have 
\begin{align*}
cl_{\mathcal{T}^+} (A)=A\cup \Phi _{d}(\mathbb{R}\setminus B)&=A\cup \left\{ x\in \mathbb{R} :\lim_{h\to 0^+} \frac{\lambda \left( (\mathbb{R}\setminus B)\cap [x-h,x +h]\right) }{2h}=1\right\}  \\
&=A\cup \left\{ x\in \mathbb{R} :\lim_{h\to 0^+} \frac{\lambda^* \left( A\cap [x-h,x +h]\right) }{2h}=1\right\} .
\end{align*}
\end{proof}

\begin{corollary}\label{uwaga}
If $A\in \mathcal{L}$, then 
\begin{enumerate}[1.]
 \item $cl_{\mathcal{T}^+} (A)=A\cup \Phi _{d} (A)$;
 \item $cl_{\mathcal{T}^+} (A)=A \Leftrightarrow \Phi _{d} (A)\subset A$; 
 \item a regular open set $A$ in $\langle \mathbb{R}, \mathcal{T}_{d} \rangle$ is a closed set in  $\langle \mathbb{R}, \mathcal{T}^+ \rangle$.
\end{enumerate}
\end{corollary}

In a generalized topological space \(\langle X, \Gamma \rangle\), we consider nowhere dense sets and strongly nowhere dense sets. A set \(A\) is nowhere dense if every set \(W\in \Gamma \setminus \{\emptyset\}\) contains a set \(V\in \Gamma \setminus \{\emptyset\}\) such that \(V\cap A =\emptyset\). In contrast, a strongly nowhere dense set \(A\) in \(\langle X, \Gamma \rangle\) satisfies the condition \(\operatorname{int}_{\mathcal{T}^+} (\operatorname{cl}_{\mathcal{T}^+} (A))=\emptyset\).
Every strongly nowhere dense set is nowhere dense, but the converse does not hold (see \cite{hw-klp}).

Sets that are countable unions of nowhere dense sets or strongly nowhere dense sets are called first category sets or strongly first category sets, respectively.
The following theorem characterizes these families.

\begin{theorem}
For every $ A\subset \mathbb{R}$ the following conditions are equivalent:
\begin{enumerate}[i)]
 \item $\forall _{W\in \mathcal{T}^+ \setminus \{\emptyset\}} \; \exists _{V\in \mathcal{T}^+ \setminus \{\emptyset\}} \; V\subset W \wedge V\cap A =\emptyset$;
 \item $int_{\mathcal{T}^+} (cl_{\mathcal{T}^+} (A))=\emptyset$;
 \item $A\in \mathbb{L}$.
\end{enumerate}
\end{theorem}

\begin{proof} 
$i)\Rightarrow ii)$ Suppose, on the contrary, that condition $ii)$ is not satisfied. 
Hence  $int_{\mathcal{T}^+} (cl_{\mathcal{T}^+} (A))\neq \emptyset$.
Let $W\in \mathcal{T}^+ \setminus \{\emptyset\}$ with $W \subset cl_{\mathcal{T}^+} (A)$.
Then there exists a set $V\subset W$, $V\in \mathcal{T}^+ \setminus \{\emptyset\}$ and $V\cap A =\emptyset$,
which contradicts the fact that $V\cap A \neq \emptyset$ since $V\subset cl_{\mathcal{T}^+} (A)$.

$ii)\Rightarrow iii)$ Suppose that $A\notin \mathbb{L}$.
Thus $cl_{\mathcal{T}^+} (A)\in \mathcal{L} \setminus \mathbb{L}$ and $int_{\mathcal{T}^+} (cl_{\mathcal{T}^+} (A))=\emptyset$.
Putting $V=cl_{\mathcal{T}^+} (A) \cap \Phi^+ (cl_{\mathcal{T}^+} (A))$ we obtain $V\in  \mathcal{T}^+\setminus \{\emptyset\}$ and $V\subset cl_{\mathcal{T}^+} (A)$. This contradicts $int_{\mathcal{T}^+} (cl_{\mathcal{T}^+} (A))=\emptyset$.

$iii)\Rightarrow i)$ If $A\in \mathbb{L}$ and $W\in \mathcal{T}^+ \setminus \{\emptyset\}$, then
$V=W\setminus A \in \mathcal{T}^+ \setminus \{\emptyset\}$ and $V\cap A =\emptyset$.
\end{proof}

\begin{corollary}
The families of nowhere dense sets, strongly nowhere dense sets, first category sets, and strongly first category sets all coincide with the family \(\mathbb{L}\).
\end{corollary}

\begin{theorem}
A set \(A\subset \mathbb{R}\) is \(\mathcal{T}^+\)-closed and \(\mathcal{T}^+\)-discrete if and only if \(A\in \mathbb{L}\).
\end{theorem}

\begin{proof} 
If \(A\in \mathbb{L}\), then for every \(B\subset A\), we have \(B\in \mathbb{L}\). Hence \(\mathbb{R} \setminus B \in \mathcal{T}^+\).
This implies that \(B\) is \(\mathcal{T}^+\)-closed. Therefore \(A\) is \(\mathcal{T}^+\)-discrete and \(\mathcal{T}^+\)-closed.

Conversely, suppose that a set \(A\) is \(\mathcal{T}^+\)-closed and \(\mathcal{T}^+\)-discrete but \(A\notin \mathbb{L}\).
Then \(A \in \mathcal{L} \setminus \mathbb{L}\). Hence there exists \(B\subset A\) such that \(B\notin \mathcal{L}\).
Since \(A\) is a \(\mathcal{T}^+\)-discrete set, we obtain that \(B\) is \(\mathcal{T}^+\)-closed, which implies \(B\in \mathcal{L}\). 
This contradiction completes the proof.
\end{proof}

The following theorems can be formulated and proved analogously to the corresponding one in \cite{hw-hl}.

\begin{theorem}
A set \( A\subset \mathbb{R}\) is \(\mathcal{T}^+\)-compact if and only if it is finite.
\end{theorem}

\begin{theorem}
The space \(\langle \mathbb{R}, \mathcal{T}^+ \rangle\) is neither first countable, nor second countable, nor separable.
\end{theorem}

Let \(\sigma (\mathcal{T}^+)\) denote the \(\sigma\)-algebra generated by the families \(\mathcal{T}^+\) and \(\mathbb{L}\).

\begin{theorem}
\(\sigma (\mathcal{T}^+)=\mathcal{L}\).
\end{theorem}

\begin{proof} 
Clearly, \(\sigma (\mathcal{T}^+)\subset \mathcal{L}\). We will show the reverse inclusion.
Let \(A\in \mathcal{L}\). Then
\[
A=(A\cap \Phi^+(A) )\cup (A\setminus \Phi^+(A) ).
\]
Since \(A\cap \Phi^+(A)\in \mathcal{T}^+\) by Property \ref{topol} and \(A\setminus \Phi^+(A)\in \mathbb{L}\) by condition 5 of Property \ref{prop}, we obtain \(A\in \sigma (\mathcal{T}^+)\).
\end{proof}

Now we determine the form of the smallest topology \(\tau(\mathcal{T}^+)\) generated by the family \(\mathcal{T}^+\).

\begin{theorem}
\(\tau (\mathcal{T}^+)=2^{\mathbb{R}}\).
\end{theorem}

\begin{proof} 
Let \(x\in \mathbb{R}\) and \(a<x<b\). Then \((a,x]\in \tau (\mathcal{T}^+)\) and \([x,b)\in \tau (\mathcal{T}^+)\).
Hence
\[
\{ x\} =(a,x]\cap [x,b)\in \tau (\mathcal{T}^+),
\]
so \(\tau (\mathcal{T}^+)=2^{\mathbb{R}}\).
\end{proof}

As stated in the theorem in \cite{hw-kt}, for a strong generalized topology \(\langle X, \Gamma \rangle\), the family
\[
\mathcal{T}_{\Gamma}=\{ A\in \Gamma : \forall _{B \in \Gamma} A\cap B \in \Gamma\}
\]
is a topology.
In the case where \(\Gamma = \mathcal{T}^+\), we obtain the following theorem.

\begin{theorem}\label{topologie}
\(\mathcal{T}_{\Gamma}=\mathcal{T}_{d}\).
\end{theorem}

\begin{proof} 
Let \(A\in \mathcal{T}_{d}\) and \(B\in \mathcal{T}^+\).
Then \(A\cap B \in \mathcal{L}\), with \(A\subset \Phi _{d}(A)\) and \(B\subset \Phi^+(B)\).
Hence \(A\cap B \subset \Phi _{d} (A) \cap \Phi^+(B)\).
By condition 8 of Property \ref{prop}, it follows that \( \Phi _{d}(A) \cap \Phi^+(B) \subset \Phi^+ (A \cap B)\).
Thus \(A\cap B \subset \Phi^+ (A \cap B)\), which implies \(A\cap B\in\Gamma\).
Therefore \(A\in \mathcal{T}_{\Gamma}\).

Now let \(A\in \mathcal{T}_{\Gamma}\). Suppose that \(A\notin \mathcal{T}_{d}\).
Hence \(A\setminus \Phi _{d}(A) \neq \emptyset\). Fix \(x\in A\setminus \Phi _{d}(A)\).
Then \(x\in \Phi^+(\mathbb{R} \setminus A)\). Thus \(x\in A\cap \Phi^+(\mathbb{R} \setminus A)\).
Since \(A\in \mathcal{T}_{\Gamma}\) and \(\Phi^+(\mathbb{R} \setminus A) \in \mathcal{T}^+\), it follows that \(A\cap \Phi^+(\mathbb{R} \setminus A)  \in \mathcal{T}^+\).
Simultaneously \( A\triangle \Phi^+(\mathbb{R} \setminus A) \in \mathbb{L}\) and \(x\in A\cap \Phi^+(\mathbb{R} \setminus A)\), so \(A\cap \Phi^+(\mathbb{R} \setminus A) \notin \mathcal{T}^+\).
Therefore \(A\cap \Phi^+(\mathbb{R} \setminus A) =\emptyset\), which contradicts \(x\in A\cap \Phi^+(\mathbb{R} \setminus A)\).
It means \(A\in \mathcal{T}_{d}\).
\end{proof}

Following the concept of approximative continuity, we introduce the notions of \(\mathcal{T}^+\)-approximative continuity and \(\mathcal{T}_{d}^+\)-approximative continuity. Recall that a function \(f: \mathbb{R} \to \mathbb{R}\) is approximatively continuous at \(x_0\) if there exists a set \(A\in \mathcal{L}\) such that \(x_0 \in \Phi _{d}(A)\) and \(f|_{A\cup \{ x_0\} }\) is continuous at \(x_0\) with respect to the natural topology.

We say that \(f: \mathbb{R} \to \mathbb{R}\) is \(\mathcal{T}^+\)-\textbf{approximatively continuous} at \(x_0\) if there exists a set \(A\in \mathcal{L}\) such that \(x_0 \in \Phi^+(A)\) and \(f|_{A\cup \{ x_0\} }\) is continuous at \(x_0\) with respect to the natural topology.

We say that \(f: \mathbb{R} \to \mathbb{R}\) is \(\mathcal{T}_{d}^+\)-\textbf{approximatively continuous} at \(x_0\) if there exists a set \(A\in \mathcal{L}\) such that \(x_0 \in \Phi^+(A)\) and \(f|_{A\cup \{ x_0\} }\) is continuous at \(x_0\) with respect to the topology \(\mathcal{T}_{d}\).

We also define \(\mathcal{T}^+\)-continuity of a function \(f: \mathbb{R} \to \mathbb{R}\) as follows: for every \(U\in \mathcal{T}_{nat}\) such that \(f(x_0)\in U\), there exists \(V\in \mathcal{T}^+\) such that \(x_0\in V\) and \(f(V)\subset U\).

Let \(f: \mathbb{R} \to \mathbb{R}\) and \(x_0 \in \mathbb{R}\). Consider the following statements:
\begin{enumerate}[i)]
 \item \(f\) is approximatively continuous at \(x_0\).
 \item \(f\) is \(\mathcal{T}^+\)-approximatively continuous at \(x_0\).
 \item \(f\) is \(\mathcal{T}_{d}^+\)-approximatively continuous at \(x_0\).
 \item \(f\) is \(\mathcal{T}^+\)-continuous at \(x_0\).
\end{enumerate}

\begin{theorem}\label{ciaglosci} 
The following implications hold: \(i)\Rightarrow ii) \Rightarrow iii)\Rightarrow iv)\).
\end{theorem}

\begin{proof}
The implications \(i)\Rightarrow ii) \Rightarrow iii)\) follow directly from the set inclusions \(\mathcal{T}_{nat}\subset \mathcal{T}_{d} \subset \mathcal{T}^+\).

We will show that \(iii)\Rightarrow iv)\). Assume that \(f: \mathbb{R} \to \mathbb{R}\) is \(\mathcal{T}_{d}^+\)-approximatively continuous at \(x_0\).
Let \(U\in \mathcal{T}_{nat}\) be such that \(f(x_0)\in U\). 
By the definition of \(\mathcal{T}_{d}^+\)-approximative continuity, there exists a set \(A\in \mathcal{L}\) such that \(x_0 \in \Phi^+(A)\) and the restriction \(f|_{A\cup \{ x_0\} }\) is \(\mathcal{T}_{d}\)-continuous at \(x_0\).
Consequently, there exists a set \(B\) such that \(x_0 \in B \in \mathcal{T}_{d}\) and \(f(B\cap (A\cup \{ x_0\}))\subset U\).
Observe from Property \ref{topol} that \(C=(A\cup \{ x_0\}) \cap \Phi^+(A\cup \{ x_0\}) \in \mathcal{T}^+\).
By Remark \ref{rem3}, we have \(B\cap C\in \mathcal{T}^+\).
Therefore \(x_0 \in B\cap C \in \mathcal{T}_{d}\) and \(f(B\cap C)\subset U\).
This shows that \(f\) is \(\mathcal{T}^+\)-continuous at \(x_0\).
\end{proof}

\begin{property}
The implication \(ii)\Rightarrow i)\) does not hold.
\end{property}

\begin{proof}
Let $A=\bigcup_{n=1}^{\infty} (a_n, b_n) \subset (0,1)$ be right-sided interval set such that $0<\bar{D}(A,0)<1$.
Define
\[
f(x)=\begin{cases}
0 & x\in A\cup \{0\} \\
1 & x\notin A\cup \{0\}
\end{cases}
\]
Then \(f\) is \(\mathcal{T}^+\)-approximatively continuous at \(0\).
However, \(f\) is not approximatively continuous at \(0\).
Indeed, suppose for contradiction that there exists a set \(E\in \mathcal{L}\) such that \(0\in \Phi _{d} (E)\) and \(f|_{E\cup \{ 0\} }\) is continuous at \(0\) with respect to the natural topology.
Observe that \(((-\delta, \delta)\cap E)\cap (\mathbb{R}\setminus A) \neq \emptyset\) for every \(\delta >0\).
But this means that \(f|_{E\cup \{ 0\} }\) is not continuous at \(0\) with respect to the natural topology.
\end{proof}

\begin{property}
The implication \(iii)\Rightarrow ii)\) holds.
\end{property}

\begin{proof}
Let \(f: \mathbb{R} \to \mathbb{R}\) be \(\mathcal{T}_{d}^+\)-approximatively continuous at \(x_0\).
Hence there exists a set \(A \in \mathcal{L}\) such that \(x_0 \in \Phi^+(A)\) and \(g=f|_{A\cup \{ x_0\} }\) is \(\mathcal{T}_{d}\)-continuous at \(x_0\).
Thus there exists a set \(B\in \mathcal{L}\) such that \(x_0 \in \Phi _{d} (B)\) and \(g|_{B\cup \{ x_0\} }\) is continuous at \(x_0\).
By condition 8 of Property \ref{prop}, we get that \(x_0 \in \Phi^+(A\cap B)\) and \(g|_{B\cup \{ x_0\} }= f|_{(A\cap B)\cup \{ x_0\} }\).
This implies that \(f\) is \(\mathcal{T}^+\)-approximatively continuous at \(x_0\).
\end{proof}

\begin{property}
The implication \(iv)\Rightarrow iii)\) does not hold.
\end{property}

\begin{proof}
Let \(E_1=\bigcup_{n=1}^{\infty} (a_n^1 , b_n^1)\) be an interval set such that 
\(0< \ldots<a_n^1<b_n^1<a_{n-1}^1<b_{n-1}^1<\ldots <a_1^1<b_1^1\) and \(\bar{D}(E_1,0)=c>0\).

If \(E_k=\bigcup_{n=1}^{\infty} (a_n^k , b_n^k)\) is already constructed for some \(k\in \mathbb{N}\), then let 
\[
E_{k+1}=\bigcup_{n=1}^{\infty} \left( \frac{1}{2}(a_n^k + b_n^k), b_n^k\right) .
\]
Obviously, \(E_1\supset E_2\supset \ldots\) and \(\bar{D}(E_k,0) =\frac{c}{2^{k-1}}\), so \(\lim\limits_{k\to \infty} \bar{D}(E_k,0) =0\).
Define
\[
f(x)=\begin{cases}
\frac{1}{k} &  \text{if } x\in E_{k}\setminus E_{k+1} \\
0 & \text{for } x=0 \\
1 & \text{for remaining } x.
\end{cases}
\]
Then \(f\) is a \(\mathcal{T}^+\)-continuous function at each \(x\in \mathbb{R}\) but is not \(\mathcal{T}_{d}^+\)-approximately continuous at \(0\).
Indeed, suppose that \(f\) is \(\mathcal{T}_{d}^+\)-approximately continuous at \(0\).  
Then there exists a set \(A\in \mathcal{L}\) such that \(\bar{D}(A,0)>0\) and a set \(B\in \mathcal{T}_{d}\) with \(0\in B\) such that 
\(\lim\limits_{x\to 0 , x\in A\cap B} f(x)=f(0)=0\).
However, \(\bar{D}(A\cap B,0)=\bar{D}(A,0)>0\), and there exists \(n_0 \in \mathbb{N}\) for which \(\frac{c}{2^{n_0-1}}<\bar{D}(A,0)\).
Hence, for each \(k\in \mathbb{N}\), there exists \(|x_k|<\frac{1}{k}\) such that \(x_k\in (A\cap B) \setminus E_{n_0}\), so \(f(x_k)>\frac{1}{n_0}\).
This contradicts the fact that \(\lim\limits_{k\to \infty} f(x_k)=f(0)=0\).
\end{proof}

\begin{corollary}
There exists a function \(f:\mathbb{R}\to \mathbb{R}\) and a point \(x_0\in \mathbb{R}\) such that \(\mathcal{T}^+\)-continuity is not equivalent to \(\mathcal{T}^+\)-approximate continuity at \(x_0\).
\end{corollary}

This corollary illustrates the distinction between \(\mathcal{T}^+\)-continuity and \(\mathcal{T}^+\)-approximate continuity. It is well known that in the case of density topology, \(\mathcal{T}_{d}\)-continuity is equivalent to \(\mathcal{T}_{d}\)-approximate continuity.
Now, we present the relationship between the measurability of a function and its various aspects of continuity appearing in Theorem \ref{ciaglosci}.

\begin{theorem}
Let \(f: \mathbb{R} \to \mathbb{R}\). Then the following conditions are equivalent:
\begin{enumerate}[i)]
 \item \(f\) is measurable.
 \item \(f\) is approximatively continuous almost everywhere (a.e.).
 \item \(f\) is \(\mathcal{T}^+\)-approximatively continuous a. e.
 \item \(f\) is \(\mathcal{T}_{d}^+\)-approximatively continuous a. e.
 \item \(f\) is \(\mathcal{T}^+\)-continuous a. e.
\end{enumerate}
\end{theorem}

\begin{proof}
The proof of the implication \((i)\Rightarrow (ii)\) can be found in \cite{hw-b}.
The implications \((ii)\Rightarrow (iii)\Rightarrow (iv) \Rightarrow (v)\) are a consequence of Theorem \ref{ciaglosci}.
We will prove that \((v)\Rightarrow (i)\). Let \(f: \mathbb{R} \to \mathbb{R}\) be \(\mathcal{T}^+\)-continuous a.e..
We will show that \(E_{\alpha}=\{ x\in \mathbb{R}:\ f(x)<\alpha \} \in \mathcal{L}\) for every \(\alpha \in \mathbb{R}\). 
Let \(F\) be the set of \(\mathcal{T}^+\)-continuity points of \(f\).
It suffices to show that \(E_{\alpha}\cap F\in \mathcal{L}\).
Observe that
\begin{equation}\label{al-new}
\forall_{x\in   E_{\alpha}\cap F} \; \exists_{E_x\in \mathcal{L}} \;  E_x\subset (E_{\alpha}\cap F) \wedge x\in \Phi^+(E_x).
\end{equation}
Indeed, if \(x_0\in E_{\alpha}\cap F\), then \(f(x_0)<\alpha\). By the \(\mathcal{T}^+\)-continuity of \(f\) at \(x_0\), there exists \(A_{x_0}\in \mathcal{T}^+\) such that 
\(x_0\in A_{x_0}\) and \(f(x)<\alpha\) for \(x\in A_{x_0}\).
Let \(E_{x_0}= A_{x_0}\cap F\). Then \(\Phi^+ ( A_{x_0})=\Phi^+ ( A_{x_0}\cap F)=\Phi^+ ( E_{x_0})\).
Since \( A_{x_0}\subset \Phi^+ ( A_{x_0})\) and \(x_0\in A_{x_0}\), we obtain that \(x_0\in \Phi^+(E_{x_0})\) and \(E_{x_0}\subset (E_{\alpha}\cap F)\).

Now let \(B\) be a Lebesgue measurable kernel of the set \( E_{\alpha}\cap F\).
From (\ref{al-new}) it follows that for any \(x\in E_{\alpha}\cap F\), there exists \(E_x\in \mathcal{L}\) such that \(E_x\subset (E_{\alpha}\cap F)\) and \(x\in \Phi^+(E_x)\). 
Moreover \(\lambda(E_x\setminus B)=0\). Hence
\(\Phi^+(E_x)\subset \Phi^+ ((E_x\setminus B)\cup B) =\Phi^+(B)\). Therefore
\[
B\subset E_{\alpha}\cap F \subset \bigcup_{x\in E_{\alpha}\cap F} \Phi^+ (E_x)\subset \Phi^+(B).
\]
Since \(\lambda(\Phi^+ (B)\setminus B)=0\), we conclude that \( E_{\alpha}\cap F \in \mathcal{L}\).
This completes the proof.
\end{proof}

\begin{theorem}
There exists a function \(f: \mathbb{R} \to \mathbb{R}\) which is \(\mathcal{T}^+\)-continuous but does not possess the Darboux property on any interval \(I\).
\end{theorem}

\begin{proof}
Let \(A_1\) and \(A_2\) be measurable subsets of \(\mathbb{R}\) such that \(A_1 \cup A_2=\mathbb{R}\), \(A_1 \cap A_2=\emptyset\), and for every interval \((a,b)\subset \mathbb{R}\), the following conditions hold:
\[ \lambda (A_1\cap (a,b))>0 \quad \text{and} \quad \lambda (A_2\cap (a,b))>0. \]
It follows that
\[ \Phi _{d}(A_1)\cap \Phi _{d}(A_2)=\emptyset. \]
Moreover, for any interval \((a,b)\subset \mathbb{R}\), we have
\[ (\Phi _{d}(A_1)\cup \Phi _{d}  (A_2))\cap (a,b)\neq (a,b). \]
This is a consequence of the fact that intervals are connected in \(\mathcal{T}_{d}\).
Therefore the sets \(\Phi _{d}(A_1)\) and \(\Phi _{d}  (A_2)\) are dense sets in \(\mathbb{R}\).

Let \(B=\mathbb{R} \setminus(\Phi _{d}(A_1)\cup \Phi _{d}  (A_2))\).
Hence \(B\) is a dense set of measure zero.
Define
\[ B_1=\{x\in B : D^+(A_1)>0\} \quad \text{and} \quad B_2=B\setminus B_1. \]
Then \(C_1=(\Phi _{d}(A_1)\cup B_1)  \in \mathcal{T}^+ \) and \(C_2=(\Phi _{d}(A_2)\cup B_2) \in \mathcal{T}^+ \).
Therefore \(C_1\cap C_2=\emptyset\) and \(C_1\cup C_2=\mathbb{R}\). The sets \(C_1\) and \(C_2\) are open and closed in \(\mathcal{T}^+\).
This implies that the function \(f(x)=\chi_{C_1}(x)\) is \(\mathcal{T}^+\)-continuous. However, for any interval \(I\), the set \(f(I)\) is not connected, which means that \(f\) does not possess the Darboux property on \(I\).
\end{proof}

\begin{theorem}
There exists a function \(f: \mathbb{R} \to \mathbb{R}\) that is \(\mathcal{T}^+\)-continuous but not Borel.
\end{theorem}

\begin{proof}
We construct a perfect, nowhere dense subset of \([0,1]\) as follows.
From \([0,1]\), remove the open interval \(J_1^1\) with center \(\frac{1}{2}\) and length \(\frac{1}{2}\).
The two remaining intervals, \(I_1^1\) and \(I_2^1\), each have length \(\frac{1}{4}\).
Next from \(I_1^1\) and \(I_2^1\) remove two concentric open intervals \(J_2^1\) and \(J_2^2\), each with length \(\frac{2}{3}\lambda (I_1^1)\).
The remaining closed intervals are denoted as \(I_i^2\), for \(i=1,2,3,4\).

This process is continued inductively. 
If \(I_i^n\), for \(i=1,2,\ldots ,2^n\), are closed intervals such that
\[ \lambda (I_i^n)=\left( \frac{1}{2} \right)^n \cdot \frac{1}{2} \cdot \frac{1}{3} \cdot \ldots \cdot \frac{1}{n+1}, \]
then from each \(I_i^n\), remove the concentric open interval \(J_i^{n+1}\) with length \(\frac{n+1}{n+2}\lambda (I_i^n)\).
The remaining closed intervals are denoted as \(I_i^{n+1}\), for \(i=1,2,\ldots ,2^{n+1}\).

Let
\[ A=\bigcup_{n=1}^{\infty} \bigcup_{i=1}^{2^{2n-2}} J_i^{2n-1}, \]
\[ B=\bigcup_{n=1}^{\infty} \bigcup_{i=1}^{2^{2n-1}} J_i^{2n}, \]
and
\[ C=[0,1] \setminus(A\cup B) \quad \text{(thus } C=\bigcap_{n=1}^{\infty} \bigcup_{i=1}^{2^{n}} I_i^{n}). \]
Obviously \(A\cap B=\emptyset\), \(\lambda (C)=0\) and \(\lambda(A)+\lambda(B)=1\).
It is easy to see that for each \(x\in C\)
\[ \bar{D}(A,x)\geq \frac{1}{2} \quad \text{and} \quad \bar{D}(B,x)\geq \frac{1}{2}. \]
Hence if \(X\subset C\), then \(A\cup X\in \mathcal{T}^+\) and \(B\cup (C\setminus X) \in \mathcal{T}^+\).
Therefore the function \(f=\chi_{A\cup X}\) is \(\mathcal{T}^+\)-continuous.
However, \(C\) is perfect, so its cardinality is that of the continuum, and there exists a set \(X\subset C\) which is not a Borel set.
Hence a \(\mathcal{T}^+\)-continuous function is not necessarily Borel measurable.
Recall that approximately continuous functions are Baire one.
\end{proof}

\begin{theorem}
The space \(\langle \mathbb{R}, \mathcal{T}^+ \rangle\) is regular.
\end{theorem}

\begin{proof}
Let \(x\in \mathbb{R}\) and \(V\in \mathcal{T}^+\) such that \(x\in V\). 
We aim to find a set \(W\in \mathcal{T}^+\) such that \(x\in W \subset \text{cl}_{ \mathcal{T}^+} (W) \subset V\).
Since \(x\in \Phi^+ (V)\), there exists a sequence \(\{ h_n \}_{n\in \mathbb{N}}\) tending to zero such that
\[ \lim_{n\to \infty} \frac{\lambda \left( V\cap [x-h_n,x +h_n]\right) }{2h_n}=\alpha >0. \]
For every \(n\in \mathbb{N}\) choose a closed set \(F_n \subset V\cap [x-h_n,x +h_n]\) such that
\[ \lambda\left(\left( V\cap [x-h_n,x +h_n]\right) \setminus F_n\right) < \lambda \left( V\cap [x-h_n,x +h_n]\right) \cdot 2h_n. \]
Then \(\lambda (F_n) >\lambda \left( V\cap [x-h_n,x +h_n]\right) \cdot (1-2h_n)\).
Define \(F=\{x\}\cup \bigcup_{n\in \mathbb{N} } F_n\). Then \(x\in F\subset V\).
We will now show that \(F\) is closed in the natural topology.
Fix a sequence \(\{ x_k \}_{k\in \mathbb{N}}\subset F\) tending to some point \(x_0\).
If infinitely many terms of the sequence \(\{ x_k \}_{k\in \mathbb{N}}\) belong to some set \(F_{n_0}\), then \(x_0\in F_{n_0}\subset F\).
If not, then there exist subsequences \(\{ n_j \}_{j\in \mathbb{N}}\) and \(\{ k_j \}_{j\in \mathbb{N}}\) such that \(x_{k_j} \in F_{n_j}\).
Hence \(|x-x_{k_j}|<2h_{n_j} \underset{j\to \infty}{\longrightarrow}0\).
Therefore \(x_{k_j} \underset{j\to \infty}{\longrightarrow} x\) and thus \(x_0=x\in F\).
Moreover,
\begin{align*}
\frac{\lambda \left( F\cap [x-h_n,x +h_n]\right) }{2h_n} &\geq
\frac{\lambda \left( F_n\cap [x-h_n,x +h_n]\right) }{2h_n} \\
&>\frac{\lambda \left( V\cap [x-h_n,x +h_n]\right) \cdot (1-2h_n) }{2h_n}.
\end{align*}
This implies that
\[ \limsup_{n\to \infty} \frac{\lambda \left( F\cap [x-h_n,x +h_n]\right) }{2h_n} \geq\alpha >0. \]
Hence \(x\in \Phi^+ (F)\). In consequence
\[ F\cap \Phi^+ (F) \subset \text{cl}_{ \mathcal{T}^+} (F\cap \Phi^+ (F)) \subset \text{cl}_{ \mathcal{T}^+} (F)\subset \text{cl}_{ \mathcal{T}_{nat}} (F)=F\subset V. \]
Setting \(W=F\cap \Phi^+ (F)\), we obtain that \(x\in W \subset \text{cl}_{ \mathcal{T}^+} (W) \subset V\).
\end{proof}

\begin{lemma}
If \(A, B \in \mathbb{L}\) and \(A \cap B = \emptyset\), then there exist \(C, D \in \mathcal{T}^+\) such that \(A \subset C\), \(B \subset D\), and \(C \cap D = \emptyset\).
\end{lemma}

\begin{proof}
Let \(G_1 \supset G_2 \supset \ldots \supset G_n \supset \ldots\) be a sequence of open (in the natural topology) sets such that \(G_n \supset A \cup B\) for each \(n \in \mathbb{N}\).
Since \(A \cup B \in \mathbb{L}\), the sets \(G_n\) may be chosen such that for each \(n \in \mathbb{N}\), if \((a,b)\) is a component of \(G_n\), then
\[ \lambda(G_{n+1}\cap (a,b))=\frac{1}{n}(b-a). \]
Let \(E=\bigcup_{n\in \mathbb{N}}( G_{4n+1} \setminus G_{4n+2})\) and
\(F=\bigcup_{n\in \mathbb{N}}( G_{4n+3} \setminus G_{4n+4})\). 
Then \((A\cup E) \cap (B\cup F)= \emptyset\).
The upper density of the set \(E\) at each point of \(A\) is equal to \(1\).
Observe that \(A \cup (\Phi _{d} (E)\cap E)\in \mathcal{T}^+\) because
\[ A\subset \Phi^+(E)=\Phi^+ (\Phi _{d} (E)\cap E)\subset \Phi^+ (A\cup (\Phi _{d} (E)\cap E)) \]
and 
\[ \Phi _{d} (E)\cap E \subset \Phi^+ (\Phi _{d} (E)\cap E)\subset \Phi^+ (A\cup (\Phi _{d} (E)\cap E)). \]
Similarly, \(B \cup (\Phi _{d} (F)\cap F)\in \mathcal{T}^+\). Setting
\(C= A \cup (\Phi _{d} (E)\cap E)\) and \(D=B \cup (\Phi _{d} (F)\cap F)\), we obtain \(C, D\in \mathcal{T}^+\), \(C\cap D=\emptyset\), \(A\subset C\), and \(B\subset D\).
\end{proof}

The idea of the above proof is due to Goffman (\cite{hw-g}).

\begin{theorem}
The space \(\langle \mathbb{R}, \mathcal{T}^+ \rangle\) is normal.
\end{theorem}

\begin{proof}
Let \(A, B\subset \mathbb{R}\) be \(\mathcal{T}^+\)-closed and disjoint sets. 
We shall find \(C, D\in \mathcal{T}^+\) such that \(A\subset C\), \(B\subset D\), and \(C\cap D=\emptyset\).
By condition 1 of Corollary \ref{uwaga}, we have \(A=A\cup \Phi _{d} (A)\) and \(B=B\cup \Phi _{d} (B)\).
Let \(A_1=(A\cap  \Phi^+ (A))\setminus \Phi _{d} (A)\) and \(A_2= A\setminus \Phi^+ (A)\).
Similarly, let \(B_1=(B\cap  \Phi^+ (B))\setminus \Phi _{d} (B)\) and \(B_2= B\setminus \Phi^+ (B)\).
Clearly \(A_1\), \(A_2\), \(B_1\), \(B_2\) are Lebesgue null sets.
Following the proof of the preceding lemma, \((\Phi _{d} (E)\cap E )\cup A_1 \in \mathcal{T}^+\) and \((\Phi _{d} (F)\cap F )\cup B_1 \in \mathcal{T}^+\).
Also, \((\Phi _{d} (E)\cap E ) \cap (\Phi _{d} (F)\cap F )=\emptyset\) because \(E\cap F=\emptyset\).

Let us consider the set \(\mathbb{R} \setminus (\Phi _{d} (B)\cup B_1)=(\mathbb{R}\setminus B)\cup B_2\).
Since \((\mathbb{R}\setminus B)\in \mathcal{T}^+\) and \(B_2\subset \Phi _{d}(\mathbb{R}\setminus B)\),
we have \(\mathbb{R} \setminus (\Phi _{d} (B)\cup B_1)\in \mathcal{T}^+\).
As \((\Phi _{d} (E)\cap E )\in \mathcal{T}_{d}\) and \((\mathbb{R} \setminus (\Phi _{d} (B)\cup B_1))\in \mathcal{T}^+\),
by condition (8) of Property \ref{prop}, we have
\[ (\Phi _{d} (E)\cap E ) \cup (\mathbb{R} \setminus (\Phi _{d} (B)\cup B_1) \subset \Phi^+\left( (\Phi _{d} (E)\cap E ) \cup (\mathbb{R} \setminus (\Phi _{d} (B)\cup B_1))\right). \]
Also, \(\Phi _{d} (A) \cup A_1\subset \Phi^+ ( \Phi _{d} (A) \cup A_1)\).
Since \(A_2\subset \Phi^+(\Phi _{d} (E)\cap E)\), we obtain that 
\[ A_2 \cup (\Phi _{d} (E)\cap E ) \subset \Phi^+ (A_2 \cup (\Phi _{d} (E)\cap E ) ). \]
Obviously, \(\Phi _{d} (A) \cup A_1 \cup A_2 = A\) and \(\Phi _{d} (B) \cup B_1 \cup B_2 = B\).
Setting
\[ C=A \cup (\Phi _{d} (E)\cap E ) \cup (\mathbb{R} \setminus (\Phi _{d} (B)\cup B_1)) \]
\[ D=B \cup (\Phi _{d} (F)\cap F ) \cup (\mathbb{R} \setminus (\Phi _{d} (A)\cup A_1)) \]
we have that \(\Phi _{d} (A) \cup A_1\in \mathcal{T}^+\), \(A_2 \cup (\Phi _{d} (E)\cap E ) \in \mathcal{T}^+\) and  
\((\Phi _{d} (E)\cap E ) \cup (\mathbb{R} \setminus (\Phi _{d} (B)\cup B_1)) \in \mathcal{T}^+\).
Therefore \(C\in \mathcal{T}^+\) and similarly \(D\in \mathcal{T}^+\).
Moreover, \(C\cap D=\emptyset\), \(A\subset C\) and \(B\subset D\), which completes the proof.
\end{proof}

\end{document}